\documentclass[12pt]{amsart}
\usepackage{amsfonts,amsmath,amssymb,lineno}
\newtheorem{theorem}{Theorem}[section]
\newtheorem{corollary}[theorem]{Corollary}

\newtheorem{example}[theorem]{Example}

\def\NN{\hbox{\sf I\kern-.13em\hbox{N}}}
\def\RR{\hbox{\sf I\kern-.14em\hbox{R}}}

\def\Cc{\hbox{\sf C\kern -.47em {\raise .48ex \hbox{$\scriptscriptstyle |$}}
   \kern-.5em {\raise .48ex \hbox{$\scriptscriptstyle |$}} }}

\newcommand{\be}{\begin{equation}}
\newcommand{\ee}{\end{equation}}

\newcommand{\cA}{{\mathcal A}}
\newcommand{\cF}{{\mathcal F}}
\newcommand{\ran}{{\rm ran}\,}

\begin{document}

\title{Triangularizability of Families of Polynomially Compact Operators}
\author{Roman Drnov\v sek, Marko Kandi\'{c}}
% \date{\thanks{ } June 27, 2005}
\date{\today}

\begin{abstract}
\baselineskip 6mm
% In this paper we consider triangularizability of families of polynomially compact operators.
A recent paper of Shemesh  shows triangularizability of a pair $\{A, B\}$ of complex matrices
satisfying the condition $A [A,B]=[A,B] B=0$, or equivalently, the matrices $A$ and $B$ commute with their product $A B$. In this paper we extend this result to polynomially compact operators on Banach spaces. The case when the underlying space is Hilbert and one of operators is normal is also studied. Furthermore, we consider families of polynomially compact operators whose iterated commutators of some fixed length are zero. We also obtain a structure result in the special case of a finite family of algebraic operators.
\end{abstract}

\maketitle

\noindent
{\it Math. Subj. Classification (2010)}: 47A15, 47A46, 47B07, 47B47, 16N20

\baselineskip 6.3mm

\section {Introduction}

The property that operators $A$ and $B$ commute with their commutator $[A,B] = A B - B A$ was studied extensively in the literature. For example, by a theorem of Kleinecke \cite{Kleinecke} and Shirokov \cite{Shirokov} the condition $[A,[A,B]]=0$ implies that the commutator $[A,B]$ is quasinilpotent. If $A$ and $B$ are algebraic elements of an associative algebra, then the assumption $[A,[A,B]]=0$ implies that $[A,B]$ is nilpotent. In the literature, this is known as Jacobson's lemma. The situation is even more interesting for operators on Hilbert spaces. Putnam \cite{Putnam} proved that whenever a normal operator $A$ on a Hilbert space commutes with a commutator $[A,B]$, then $A$ and $B$ commute. This result can be, in particular, applied to complex matrices. In this special case the assumption that $A$ is normal can be replaced with the weaker assumption that $A$ is diagonalizable. Shapiro  \cite[Theorem 10]{Shapiro} proved that in this case the condition $[A,[A,B]]=0$ together with diagonalizability of $A$ imply that $A$ and $B$ commute.

The main goal of this paper is to prove that families of polynomially compact operators, which are in some sense close to commutative ones, have plenty invariant subspaces. Probably the most known result of this type, which is proved by an application of Lomonosov's theorem, states that every commutative family of polynomially compact operators on a complex Banach space is triangularizable.
We will extend this result (see Theorem \ref{family_polynomially_compact}) to families of polynomially compact operators  with the property for which there exists a positive integer $n$ such that all iterated commutators of length at least $n$ are zero.  This result improves Murphy's result \cite{Murphy} which was proved only in the case for compact operators $A$ and $B$ which commute with their commutator.

On the other hand, the property that operators $A$ and $B$ commute with their product $AB$ was not studied so far in the literature in great extent.
A recent paper of Shemesh \cite{Shemesh} shows triangularizability of a pair $\{A, B\}$ of complex matrices satisfying the condition $A [A,B]=[A,B] B=0$, or equivalently, the matrices $A$ and $B$ commute with their product $A B$. We extend this result to polynomially compact operators on Banach spaces. We also consider algebras generated by finite families of algebraic operators.

This paper is organized as follows. In Section 2, we gather relevant definitions and notation needed throughout the text. In Section 3, we consider triangularizability of families of polynomially compact operators. We also study general algebraic properties of algebras generated by them. In Section 4, we extend the above-mentioned result of Shemesh to polynomially compact operators on a complex Banach space. We also consider a special case when the underlying space is a Hilbert space and one of the operators is normal.

\section {Preliminaries}

Let $V$ be a complex vector space.
By an {\it operator} on a vector space $V$ we mean a linear transformation from $V$ into itself.
An operator on $V$ is {\it scalar} if it is a scalar multiple of the identity operator $I$ on $V$.
Otherwise, it is called {\it nonscalar}.
The notation $[S,T]$ is used as an abbreviation for the commutator $S T - T S$,
where $S$ and $T$ are operators on $V$.
An operator $T$ on $V$ is said to be {\it algebraic} if there exists a nonzero complex polynomial $p$
such that the operator $p(T) = 0$. The monic polynomial of smallest degree with this property
is called the {\it minimal polynomial} of an algebraic operator $T$.
By $\deg p$ we denote the degree of a polynomial $p$.
An operator $T$ is said to be {\it  nilpotent} if $T^n = 0$ for some positive integer $n$.
The smallest $n$ with this property is called the {\it nilpotency index} of  $T$, and it is denoted by $n(T)$.

A subspace $M$ of $V$ is {\it invariant} under an operator $T$ on $V$
whenever $T (M) \subseteq M$.  A subspace $M$ of $V$ is {\it nontrivial} if $\{0\} \neq M \neq V$.
Let ${\cF}$ be a family of operators on $V$. A subspace $M$ of $V$ is {\it invariant} under ${\cF}$
if $M$ is invariant under every $T \in \cF$.
If, in addition, the subspace $M$ is invariant under every operator $S$ that commutes with all operators of $\cF$,
$M$ is said to be {\it hyperinvariant} under $\cF$. A family $\cF$ is {\it algebraically reducible} if there exists a nontrivial subspace invariant under $\cF$.
A family $\cF$ is {\it algebraically triangularizable} if there is a chain of invariant subspaces for $\cF$
which is maximal as a subspace chain.
When $V$ is a Banach space, we replace subspaces  by closed subspaces and operators by bounded operators
to obtain the corresponding definitions of {\it reducibility},  {\it triangularizability} and {\it hyperinvariance}.
For details, see the monograph \cite{RadRos}.

By $\mathcal B(X)$ we denote the Banach algebra of all bounded operators on a Banach space $X$,
and by $\mathcal K(X)$ its subalgebra of all compact operators.
An operator $T \in \mathcal B(X)$ is said to be {\it polynomially compact}
if there exists a nonzero complex polynomial $p$ such that the operator $p(T)$ is compact.
The monic polynomial of smallest degree with this property is called the {\it minimal polynomial} of a polynomially compact operator $T$. Note that an operator $T \in \mathcal B(X)$ is polynomially compact if and only if the canonical projection from $\mathcal B(X)$ to the Calkin algebra $\mathcal B(X)/\mathcal K(X)$ maps $T$ to an algebraic element.

The following well-known theorem of Lomonosov \cite{Lom} is essential for our results.

\begin {theorem} \label{Lomonosov}
A nonscalar operator on an infinite-dimensional complex Banach space which commutes with a nonzero compact operator has a nontrivial hyperinvariant closed subspace.
\end {theorem}

An application of Theorem \ref{Lomonosov} gives the following result (see, e.g., \cite{konva}).

\begin {corollary}  \label{Konvalinka}
A nonscalar polynomially compact operator on a complex Banach space
has a nontrivial hyperinvariant closed subspace.
\end {corollary}

Corollary \ref{Konvalinka} and the Triangularization Lemma (see, e.g., \cite[Lemma 7.1.11]{RadRos}) imply the following
triangularizability result.

\begin {corollary}  \label{commutative}
Every commutative family of polynomially compact operators is triangularizable.
\end {corollary}

Let $\mathcal F$ be a family of operators on $V$. Define ${\mathcal F}^{[0]} = \mathcal F$, and
${\mathcal F}^{[k]} = \{ [A, B] : A \in {\mathcal F}^{[k-1]}, B\in  \mathcal F \}$ for $k =1, 2, \ldots$.
If ${\mathcal F}^{[k]}=\{0\}$ for some $k \in \mathbb N$, the family $\cF$ is said to be {\it L-nilpotent}. In particular, if ${\mathcal F}^{[1]}=\{0\}$, then operators from $\mathcal F$ commute.
Note that every nilpotent Lie algebra is L-nilpotent.
If there exists $k \in \mathbb N$ such that $T_1 T_2 \ldots T_k  = 0$
for every choice of $T_1$, $T_2$, $\ldots$, $T_k$ in $\cF$,  the family $\cF$ is said to be  {\it nilpotent}.
Clearly, if $\cF$ is a nilpotent family, then it is L-nilpotent as well.
We say that $\cF$ is {\it locally (L-)nilpotent} if every finite subset of $\cF$ is (L-)nilpotent. An algebra $\mathcal A$ is said to be {\it locally finite} if every finitely-generated subalgebra of $\mathcal A$ is finite-dimensional. The {\it Jacobson radical} of an algebra $\mathcal A$, denoted by $J(\mathcal A)$, is defined as the intersection of all primitive ideals of $\mathcal A$. If the algebra $\mathcal A$ is finite-dimensional, then $J(\mathcal A)$ is the unique maximal nilpotent ideal of $\mathcal A$. Although, in general, the algebra $\mathcal A$ does not contain maximal nilpotent ideals, $J(\mathcal A)$ contains left and right ideals of $\mathcal A$ which consist of nilpotent elements (see, e.g., \cite{Bresar}). \\

\section {L-nilpotent families}

By Corollary \ref{commutative}, every commutative family of polynomially compact operators is triangularizable. We extend this result to L-nilpotent families of polynomially compact operators.

\begin {theorem} \label{family_polynomially_compact}
Let $\mathcal F$ be a L-nilpotent family of polynomially compact operators on a complex Banach space $X$.
Then the family $\mathcal F$ is triangularizable.
\end {theorem}

\begin {proof}
By the Triangularization Lemma  it suffices to prove that the family $\mathcal F$ is reducible.
We may assume that $\mathcal F$ contains a nonscalar operator.  If ${\mathcal F}^{[1]}=\{0\}$, then
$\mathcal F$  is reducible, by Corollary \ref{commutative}. So, we may assume that
${\mathcal F}^{[k]}=\{0\}$ and ${\mathcal F}^{[k-1]} \neq \{0\}$ for some $k \geq 2$. Therefore, there exist operators
$A \in {\mathcal F}^{[k-2]}$ and  $B \in {\mathcal F}$ such that $[A, B] \neq 0$ and
every operator in $\mathcal F$ commutes with $[A, B]$. Furthermore, it follows from
Wielandt's theorem \cite{Wielandt} that $[A, B]$ is not a scalar operator.

Suppose first that there exists an operator $C$ in $\mathcal F$ such that $C$ is not algebraic. Then there exists a polynomial $p$ such that $p(C)$ is a nonzero compact operator, and the space $X$ is necessarily
infinite-dimensional.  Since $[[A, B], C] = 0$,  the compact operator $p(C)$ commutes with
the nonscalar operator $[A, B]$, and so $[A, B]$ has a nontrivial hyperinvariant subspace, by Theorem \ref{Lomonosov}. Since every operator in $\mathcal F$ commutes with $[A, B]$, it follows that the family $\mathcal F$  is reducible.

Suppose now that $\mathcal F$ consists of algebraic operators. In this case the commutator $[A, B]$
is nilpotent, by Jacobson's lemma (see, e.g., \cite{Bracic}). Since every operator in $\mathcal F$ commutes with $[A, B]$, the kernel $\ker([A, B])$ is a nontrivial closed subspace that is invariant under $\mathcal F$. This completes the proof.
\end {proof}

Triangularizable families of operators on finite-dimensional vector spaces have plenty of important and interesting properties. As an example, if $S$ and $T$ are simultaneously triangularizable, then their commutator $[S,T]$ is nilpotent. This is often proved by representing $S$ and $T$ by upper triangular matrices with respect to some basis of the underlying space. In general, in infinite-dimensional vector spaces we cannot expect the same idea to work. However, in the case of finite L-nilpotent families of algebraic operators, operators can be still represented by upper triangular finite operator matrices with scalar operators as their diagonal entries.

\begin {theorem}\label{algebraic operators}
Let $\mathcal F$ be a finite L-nilpotent family of algebraic operators on a complex vector space $V$,
and let $\mathcal A$ be the algebra generated by $\cF$.
Then there exists a finite direct decomposition of $V$ with respect to which every member of $\cA$ has
an upper triangular block form having scalar operators for diagonal blocks.
\end {theorem}

\begin {proof}
Consider first the case when  ${\mathcal F}^{[1]}=\{0\}$.
Let us denote by $d(T)$ the degree of the minimal polynomial of an algebraic operator $T$.
We will use the induction on the sum $s(\cF) := \sum_{T \in \cF} d(T)$.
If  $\cF$ consists of scalar operators, i.e., $s(\cF)$ is equal to the cardinality of $\cF$,
then the conclusion of the theorem clearly holds. So, assume that there is a nonscalar operator $A \in \cF$.

Assume first that the minimal polynomial of $A$ is of the form $(\lambda - \lambda_0)^m$,
where $\lambda_0$ is the only eigenvalue of $A$ and $m \ge 2$. Then $V_1 := \ker (A-\lambda_0)^{m-1}$ is
a nontrivial subspace of $V$ that is invariant under every member of $\cA$. Choosing any vector space complement $V_2$ of $V_1$, every member of $\cA$ has an upper triangular $2 \times 2$ block form, with respect to the decomposition $V = V_1 \oplus V_2$. Furthermore, the minimal polynomial of the restriction of $A$ to $V_1$ is equal to
$(\lambda - \lambda_0)^{m-1}$, and the range of the operator $A-\lambda_0$ is contained in $V_1$, so that
the compression of $A$ to $V_2$  is a scalar operator.
If,  for each $i=1,2$, we consider the family $\cF|_{V_i}$ consisting from all compressions of members of $\cF$ to $V_i$, then we have $s(\cF|_{V_i}) < s(\cF)$, so that we can use the induction hypothesis to $\cF|_{V_i}$, and consequently we obtain the conclusion of the theorem for $\cF$.

Assume now that $A$ has at least $2$ eigenvalues, so that  the minimal polynomial of $A$ is of the form
$(\lambda - \lambda_1)^m \, q(\lambda)$, where $\lambda_1$ is an eigenvalue of $A$, $m \ge 1$,
and $q(\lambda)$ is a nontrivial polynomial not having $\lambda_1$ as one of its zeros.
Then $V_1 := \ker (A-\lambda_1)^m$ is a nontrivial subspace of $V$ that is invariant under every member of $\cA$. Choosing any vector complement $V_2$ of $V_1$, every member of $\cA$ has an upper triangular
$2 \times 2$ block form, with respect to the decomposition $V = V_1 \oplus V_2$. Furthermore, the minimal polynomial of the restriction of $A$ to $V_1$ is equal to $(\lambda - \lambda_1)^m$, and the minimal polynomial of the compression of $A$ to $V_2$ divides the polynomial $q(\lambda)$. If,  for each $i=1,2$, we consider the family
 $\cF|_{V_i}$ consisting from all compressions of members of $\cF$ to $V_i$, then we also have
$s(\cF|_{V_i}) < s(\cF)$, and so we can use the induction hypothesis again to get the conclusion of the theorem for $\cF$. This completes the proof in the case ${\mathcal F}^{[1]}=\{0\}$.

Consider now the general case. To each family $\cF$ satisfying the assumptions of the theorem, we
associate a pair of two integers as follows. If ${\mathcal F}^{[1]}=\{0\}$, put $c(\cF) := (1, 1)$. Otherwise,
define $c(\cF) := (k, \sum n([A,B]))$, where $k \ge 2$ is determined by the conditions ${\mathcal F}^{[k]}=\{0\}$ and ${\mathcal F}^{[k-1]} \neq \{0\}$, and the sum runs over all $A \in {\mathcal F}^{[k-2]}$ and  $B \in {\mathcal F}$
with $[A,B] \neq 0$.  Note that, by Jacobson's lemma, the commutator $[A, B]$ (in the sum) is necessarily nilpotent, since $B$ commutes with it. We will prove the theorem by induction over the set of
all possible values of $c(\cF)$  that is a subset of ${\mathbb N} \times {\mathbb N}$, ordered lexicographically.

The base case of induction holds by the first part of the proof.
Assume therefore that $\cF$ is a family satisfying the assumptions of the theorem,
and we have ${\mathcal F}^{[k]}=\{0\}$ and ${\mathcal F}^{[k-1]} \neq \{0\}$ for some  $k \ge 2$.
Choose $A \in {\mathcal F}^{[k-2]}$ and  $B \in {\mathcal F}$ such that $[A,B] \neq 0$.
Then the image $V_1 := \ran([A, B])$ is a nontrivial subspace of $V$ that is invariant under every member of $\cA$. Choosing any vector space complement $V_2$ of $V_1$, every member of $\cA$ has an upper triangular
$2 \times 2$ block form, with respect to the decomposition $V = V_1 \oplus V_2$.
Furthermore, the nilpotency index of  the restriction of $[A,B]$ to $V_1$ is smaller than $n([A,B])$,
and the compression of $[A,B]$ to $V_2$ is a zero operator. It follows that, for given $i=1,2$,
the family $\cF|_{V_i}$ consisting from all compressions of members of $\cF$ to $V_i$ has the property that
$c(\cF|_{V_i}) < c(\cF)$. Therefore,  we can use the induction hypothesis
to obtain the conclusion of the theorem for $\cF$. This completes the proof.
\end{proof}

As a consequence of Theorem \ref{algebraic operators},
we can establish the following algebraic properties of the algebra $\cA$.

\begin {corollary}  \label{algebraic algebra}
Under the assumptions of Theorem \ref{algebraic operators},
the algebra $\cA$ is a finite-dimensional algebra of algebraic operators,  its Jacobson radical $J(\mathcal A)$ is nilpotent,
and the quotient algebra $\mathcal A/J(\mathcal A)$ is isomorphic to $\mathbb C^n$ for some $n \in \mathbb N$.
\end {corollary}

\begin {proof}
By Theorem \ref{algebraic operators}, there is a direct decomposition $V = V_1 \oplus \cdots \oplus V_m$
with respect to which every member of $\cA$ has an upper triangular block form
having scalar operators for diagonal blocks. For $i=1, \ldots, m$, we denote by $P_i$ the projector on $V_i$.
If $\cF = \{A_1, \ldots, A_r\}$ then, for each  $j=1, \ldots, r$, there are uniquely determined complex numbers
$\lambda_{1,j}, \ldots, \lambda_{m,j}$ such that $A_j = \sum_{i=1}^m \lambda_{i,j} P_i + N_j$, where
$N_j$ is a nilpotent operator having strictly upper triangular block form, with respect to the decomposition above.
Then the algebra $\cA$ is contained in the linear span of the set
$$ \{ P_{i_1} N_{j_1} P_{i_2} N_{j_2}P_{i_3} N_{j_3} \cdots P_{i_{k}} N_{j_{k}} P_{i_{k+1}} , k=0, 1, 2, \ldots, m,
1 \le i_s \le m, 1 \le j_s \le r \} . $$
It follows that the algebra $\cA$ is finite-dimensional, and so its members are algebraic operators.
The Jacobson radical $J(\mathcal A)$ is exactly the set of all nilpotent operators from $\cA$.
Then it is not difficult to show that the homomorphism $p : \cA \rightarrow \mathbb C^m$ defined for generators
by $p(A_j) = (\lambda_{1,j}, \lambda_{2,j}, \ldots, \lambda_{m,j})$ gives an isomorphism
of $\mathcal A/J(\mathcal A)$ and $\mathbb C^n$ for some $n \le m$.
We complete the proof with a comment that the classical Wedderburn theory could be also used here.
\end {proof}

Guinand \cite{Guinand82} constructed a pair $\{A,B\}$ of weighted shifts on the Hilbert space $l^2$ such that the semigroup generated by them consists of nilpotent operators of nilpotency index $3$. Moreover, the sum $A+B$ is equal to the forward shift, so that, in particular, the algebra generated by $A$ and $B$ is infinite-dimensional. This example shows that an algebra generated by a finite triangularizable family of algebraic operators is not necessarily finite-dimensional.

We proceed by exploring algebraic properties of locally L-nilpotent families of algebraic operators. Hadwin et al. \cite{HNRRR86} constructed a locally nilpotent non-nilpotent associative algebra $\mathcal A$ of nilpotent operators on a separable Hilbert space that is not reducible. This example shows that Theorem \ref{family_polynomially_compact}, in general, does not hold for locally L-nilpotent families of polynomially compact operators. However, as an application of Hadwin's result \cite{Hadwin} the algebra $\mathcal A$ is algebraically triangularizable. We prove next that every locally L-nilpotent family of algebraic operators on a complex vector space is algebraically triangularizable.

\begin {corollary}\label{lokalno nilpotentna Lmnozica}
Let $\mathcal F$ be a locally L-nilpotent family of algebraic operators on a complex vector space,
and let $\mathcal A$ be the algebra generated by $\mathcal F$.
Then the following statements hold for the algebra $\mathcal A$.
\begin {enumerate}
\item [(a)] $\mathcal A$ is locally finite.
\item [(b)] $\mathcal A/J(\mathcal A)$ is commutative.
\item [(c)] $J(\mathcal A)=\{T\in \mathcal A:\, T \textrm{ is nilpotent}\}.$
\item [(d)] $\mathcal A$ is algebraically triangularizable.
\end {enumerate}
\end {corollary}

\begin {proof}
To prove (a), let $A_1,\ldots,A_n$ be arbitrary operators from $\mathcal A$, and let $\mathcal F_0$ be a finite subset of $\mathcal F$ such that the associative algebra $\mathcal A_0$ generated by $\mathcal F_0$ contains $A_1,\ldots,A_n.$
Since the family $\mathcal F$ is locally L-nilpotent, the algebra $\mathcal A_0$ is finite-dimensional, by Corollary \ref{algebraic algebra}.

To prove (b), we must show that commutators of operators from $\mathcal A$ are contained in the radical $J(\mathcal A)$ of $\mathcal A$. In order to prove this, let us choose arbitrary operators $A$ and $B$ in $\mathcal F$, and an arbitrary operator $T$ from the ideal generated by $AB-BA$. Then there exist $\lambda\in\mathbb C$, $k\in\mathbb N$ and operators $C,D,E_1,\ldots,E_k,F_1,\ldots,F_k\in \mathcal A$ such that  $$T=\lambda(AB-BA)+C(AB-BA)+(AB-BA)D+\sum_{i=1}^k E_i(AB-BA)F_i.$$ Let $\mathcal F_0$ be a finite subset of $\mathcal F$ such that the algebra $\mathcal A_0$ generated by $\mathcal F_0$ contains operators $A,B,C,D,E_1,\ldots,E_k,F_1,\ldots,F_k.$ Since the algebra $\mathcal A_0/J(\mathcal A_0)$ is commutative and $J(\mathcal A_0)$ is a nilpotent ideal of $\mathcal A_0$ by Corollary \ref{algebraic algebra}, the operator $T$ is contained in $J(\mathcal A_0)$, and is therefore nilpotent. This implies that the ideal in $\mathcal A$ generated by the operator $AB-BA$ consists of nilpotent operators, which implies that  $AB-BA \in J(\mathcal A)$.

Since the algebra $\mathcal A$ is locally finite, it consists of algebraic operators, so that (b) together with  \cite[Theorem 2.4]{Hadwin} give (c) and (d).
\end {proof}

Although locally nilpotent families of polynomially compact operators are not necessarily triangularizable,
this is the case under the additional assumption that operators essentially commute, that is, their images in the Calkin algebra
commute.

\begin {corollary}\label{bistveno komutativna}
Locally L-nilpotent family $\mathcal F$ of essentially commuting polynomially compact operators on a complex Banach space $X$ is triangularizable. Furthermore, if $\mathcal A$ is the associative algebra generated by $\mathcal F$, then the Banach algebra $\overline{\mathcal A}/J(\overline{\mathcal A})$ is commutative.
\end {corollary}

\begin {proof}
We first claim that the associative algebra $\mathcal A$ generated by $\mathcal F$ is essentially commutative algebra of polynomially compact operators. In order to prove this, let $\pi:\mathcal B(X)\to \mathcal B(X)/\mathcal K(X)$ be the canonical projection from $\mathcal B(X)$ to the Calkin algebra $\mathcal B(X)/\mathcal K(X)$. Since the algebra $\pi(\mathcal A)$ is generated by $\pi(\mathcal F)$, the algebra $\pi(\mathcal A)$ is commutative. By \cite[Lemma 1.6]{konva} the algebra $\pi(\mathcal A)$ consists of algebraic elements. Therefore, the algebra $\mathcal A$ is essentially commutative and consists of polynomially compact operators which proves the claim.

Choose arbitrary $A$ and $B\in \mathcal A$. Then there exists a finite subset $\mathcal F_0\subseteq \mathcal F$ such that $A$ and $B$ are contained in the associative algebra generated by $\mathcal F_0$. Since $\mathcal F$ is a locally L-nilpotent family
of polynomially compact operators, the set $\mathcal F_0$ is triangularizable by Theorem \ref{family_polynomially_compact}. From this we conclude that the pair $\{A,B\}$ is  triangularizable, so that the set $\mathcal F$ is triangularizable by \cite[Theorem 2.10]{konva} (see also \cite{KanLMA}). The remaining part of the corollary follows from \cite[Theorem 3.5]{KanLMA}.
\end {proof}

\section{Operators $A$ and $B$ commute with the product $AB$}

Theorem \ref{family_polynomially_compact} can be applied to the special case when $\mathcal F$ consists of two operators $A$ and $B$. This implies that whenever polynomially compact operators $A$ and $B$ on a Banach space satisfy
$[A,[A,B]]=[B,[A,B]]=0$, then $A$ and $B$ are simultaneously triangularizable.
Recently, Shemesh \cite{Shemesh} proved that complex $n\times n$ matrices $A$ and $B$ are simultaneously triangularizable whenever $A[A,B]=[A,B]B=0$, or equivalently,  $A$ and $B$ commute with their product $AB$.
We extend this result to polynomially compact operators.

\begin {theorem}\label{poly compact}
Let $A$ and $B$ be polynomially compact operators on a complex Banach space $X$.
If $A [A,B]=[A,B] B=0$, then $A$ and $B$ are simultaneously triangularizable.
\end {theorem}

\begin {proof}
By the Triangularization Lemma, it suffices to prove reducibility of the pair $\{A,B\}$.
In view of Corollary  \ref{Konvalinka} we may assume that neither $A$ nor $B$ is a scalar operator and
that $AB \neq BA$.  We consider two cases.
\medskip \\
\noindent{\it Case 1:} The product $AB$ is not a scalar operator.

Assume first that $A$ and $B$ are algebraic operators with the minimal polynomials $p_A$ and $p_B$, respectively.
Since $A$ and $B$ commute with $A B$, we have $(A B)^n=A^n B^n$ for all $n \in \mathbb N$, and so
the set $\{(A B)^n :  n \in \mathbb N\}$ is contained in the finite-dimensional space spanned by
$\{A^j B^k : 0 \le j <  \deg p_A, 0 \le k <  \deg p_B\}$. It follows that the operator $AB$ is algebraic.
Since it is not a scalar operator, there exists an eigenspace of $AB$ that is a nontrivial closed subspace of $X$.
Since every eigenspace is hyperinvariant, it follows that the pair $\{A,B\}$ is reducible.

Assume now that at least one of the operators $A$ and $B$ is not algebraic. Then the space $X$ is
necessarily infinite-dimensional, and there is a nonzero compact operator $K$ obtained as a polynomial of
either $A$ and $B$ such that $K$ commutes with $A B$. Therefore, the operator $A B$ has
a nontrivial hyperinvariant closed subspace, by Theorem \ref{Lomonosov}. This proves that
the pair $\{A,B\}$ is reducible.
\medskip \\
\noindent{\it Case 2:} The product $AB$ is a scalar operator.

If $AB=0$, then the range of $B$ is not dense in $X$, and so its closure is a nontrivial closed subspace
invariant under both $A$ and $B$. Hence, we may assume that $AB \neq 0$.
Multiplying $A$ by a suitable complex number if necessary, we may assume further that  $AB= I$.
It is easy to see that the range $\ran B$ is a closed subspace.
If $\ran B = X$ then $B$ is invertible and $A=B^{-1}$, so that $A$ and $B$ commute, contradicting
the assumption from the beginning of the proof. Therefore, the closed subspace $\ran B$ is proper.

Assume first that the operator $B$ is algebraic. Since $AB=I$, $B$ cannot be nilpotent, so that there
exists an eigenvalue $\lambda  \neq 0$ such that the corresponding eigenspace
$\ker(B-\lambda I)$ is a proper closed subspace of $X$.  If $x \in \ker(B-\lambda I)$, then
$0 = A(B-\lambda I) x = x - \lambda Ax$, and so  $Ax = 1/\lambda \cdot x  \in \ker (B-\lambda I)$ which proves that $\ker(B-\lambda I)$ is  invariant under $A$. Since it is also invariant under $B$, the pair $\{A,B\}$ is reducible.

Assume now that the operator $B$ is not algebraic. Then the space $X$ is necessarily infinite-dimensional.
Let $p$ be the minimal polynomial of the polynomially compact operator $B$, so that $K := p(B)$ is a nonzero compact operator. Define a number $c$ and a polynomial $q$ by $c := p(0)$ and $q(z) := (p(z)-c)/z$.
We claim that $c \neq 0$. Assume otherwise that  $c=0$. Since $A B=I$, the operator $B$ is not
compact, and so $\deg p \ge 2$.
Since the operator $q(B)= A p(B)=A K$ is also compact, this contradicts the minimality of $p$,
and so the claim is proved.
Since the operator $K - c I$ has a finite descent, there exists a positive integer $m_0$ such that
$\ran((K-cI)^j)=\ran((K-cI)^{m_0})$ for all $j \geq m_0$.
Since $B^n q(B)^n = (B q(B))^n = (p(B)-cI)^n = (K-cI)^n$ for  every $n\in \mathbb N$, we have
$$ \ran (K-c I)^{m_0}=\ran (K-cI)^j \subseteq \ran (B^j) $$
for all $j\geq m_0$. Since $K$ is compact and $c\neq 0$, the quotient Banach space $X/\ran(K-cI)^{m_0}$ is finite-dimensional. Therefore, there exists a positive integer $n_0$ such that $\ran (B^j)=\ran (B^{n_0})$
for all $j\geq n_0$ and the codimension of $\ran (B^{n_0})$ is finite.
The subspace $\ran (B^{n_0})$ is obviously invariant under the operator $B$.
If $x\in\ran (B^{n_0})=\ran (B^{n_0+1})$, then there exists $y \in X$ such that $x=B^{n_0+1} y$, and so   $Ax=B^{n_0}y \in \ran (B^{n_0})$.  This shows that  $\ran (B^{n_0})$ is also invariant under $A$.
Since $\ran (B^{n_0})$ is a finite-codimensional subspace that is contained in a proper subspace $\ran B$,
it is a nontrivial closed subspace of $X$. This proves that the pair $\{A,B\}$ is reducible.
\end {proof}

Shemesh \cite{Shemesh} also proved that the pair $\{A,B\}$ of complex $n\times n$ matrices is triangularizable whenever $A[A,B]=B[A,B]=0$. We provide an extension of this result to polynomially compact operators. 

\begin {theorem}
Let $A$ and $B$ be polynomially compact operators on a complex Banach space $X$.
If $A [A,B]=B [A,B]=0$, then $A$ and $B$ are simultaneously triangularizable.
\end {theorem}

\begin {proof}
Similarly as in the proof of Theorem \ref{poly compact} we only need to prove reducibility of the pair $\{A,B\}$ of nonscalar noncommuting polynomially compact operators. In this case, the closure of the range of the operator $AB-BA$ is a nontrivial closed subspace of $X$ which is invariant under both $A$ and $B$.
\end {proof}

 So far, we considered operators acting on Banach spaces. When they act on Hilbert spaces, we can say more. Suppose that $A$ and $B$ are polynomially compact operators on a Hilbert space $\mathcal H$. If $A$ is normal and $[A,[A,B]]=0$, then  by Putnam's theorem \cite[Theorem III]{Putnam}  $A$ and $B$ commute, and so they are simultaneously triangularizable.
Similar conclusion (see Theorem \ref{normalnost}) holds whenever $A$ is normal and $A[A,B]=0$.
However, the following easy example shows that, in general, $A$ and $B$ do not commute.

\begin {example}
By a direct calculation one can verify that simultaneously triangularizable matrices
$$A=\left[\begin {array}{cc}
1&0\\
0&0
\end {array}\right] \qquad \textrm{and} \qquad
B=\left[\begin {array}{cc}
0 & 0\\
1 & 0
\end {array}\right]$$ satisfy $A[A,B]=0$, while
$$ AB-BA=\left[\begin {array}{cc}
0&0\\
-1&0
\end {array}\right].$$
\end {example}

\begin {theorem}\label{normalnost}
Let $A$ and $B$ be polynomially compact operators on a Hilbert space $\mathcal H$. Suppose that
$A$ is normal and that $A [A,B]=0$.
Then $A$ and $B$ are simultaneously triangularizable and $[A,B]^2=0$.
If $B$ is also normal, then $A$ and $B$ are simultaneously diagonalizable, and so they commute.
\end {theorem}

\begin {proof}
Clearly, we may assume that $A$ and $B$ are not scalar operators.
If $A$ is injective, the assumption $A [A,B]=0$ implies that $[A,B]=0$, and so $A$ and $B$ are simultaneously triangularizable.
Therefore, we may assume that the kernel $\ker A$ is a nontrivial closed subspace of $\mathcal H$.
Since the operator $A$ is normal, $A$ and $B$ can be represented as
$$ A=\left[\begin {array}{cc}
0&0\\
0&A_{22}
\end {array}
\right]\qquad \textrm{and}\qquad
B=\left[\begin {array}{cc}
B_{11}&B_{12}\\
B_{21}&B_{22}
\end {array}
\right] , $$
with respect to the decomposition $\mathcal H=\ker A \oplus (\ker A)^\perp = \ker A \oplus \overline{\ran A}$.
Since $A^2 B = A B A$, we obtain that $A_{22}^2 B_{21} = 0$ and $A_{22} (A_{22} B_{22} - B_{22} A_{22}) = 0$,
and so  $B_{21} = 0$ and $A_{22} B_{22} = B_{22} A_{22}$, as $A_{22}$ is injective.
Now it is easy to see that we have $[A,B]^2=0$.
Since commuting polynomially compact operators $A_{22}$ and $B_{22}$ are  simultaneously triangularizable and the polynomially compact operator $B_{11}$ is  triangularizable, we conclude that $A$ and $B$ are simultaneously triangularizable.

Assume now that $B$ is also normal.
By \cite[Corollary 1.24, p.\,22]{INVRadRos}, it is also completely normal, i.e., every its invariant closed subspace $M$  is reducing (that is, $M^\perp$ is also invariant). It follows that $B_{12} = 0 $ in the above
representation of $B$. Therefore, $B$ is block diagonal, and so $A$ and $B$ commute.
By \cite[Proposition 1.10, p.\,24]{INVRadRos}, the operator $A$ is diagonalizable, that is, the set of its eigenvectors spans the space $\mathcal H$.  So, there exist $N>1$
(where we do not exclude the possibility $N=\infty$), pairwise different complex numbers $\{\lambda_k\}_{k=1}^N$ and pairwise orthogonal projections $\{P_k\}_{k=1}^N$ such that $A=\bigoplus\limits_{k=1}^N \lambda_k P_k$.
Since $AB=BA$, the range of each projection $P_k$ is a reducing subspace for the operator $B$.
Since the restriction of $B$ to the range of $P_k$ is polynomially compact and normal, it is diagonalizable by
\cite[Proposition 1.10, p.\,24]{INVRadRos}. From here it follows that $A$ and $B$ are simultaneously diagonalizable.
\end {proof}

We complete the paper by the dual version of  Theorem \ref{normalnost}.

\begin {theorem}
Let $A$ and $B$ be polynomially compact operators on a Hilbert space $\mathcal H$. Suppose that
$B$ is normal and that $[A,B]B=0$.
Then $A$ and $B$ are simultaneously triangularizable and $[A,B]^2=0$.
If $A$ is also normal, then $A$ and $B$ are simultaneously diagonalizable, and so they commute.
\end {theorem}

\begin {proof}
Since $B^*[B^*,A^*]=([A,B]B)^*=0$, the pair $\{A^*,B^*\}$ of polynomially compact operators is triangularizable and $[B^*,A^*]^2=0$,  by Theorem \ref{normalnost}. The latter implies that  $[A,B]^2=0$.
Let $\mathcal C$ be one of triangularizing chains for the pair $\{A^*,B^*\}$. Then the set $\mathcal C^\perp:=\{M^\perp:\; M\in \mathcal C\}$ is a chain of closed subspaces invariant under $A$ and $B$.
Since the mapping $N\mapsto N^\perp$ is an anti-isomorphism of the lattice of all closed subspaces of $\mathcal H$, the chain $\mathcal C^\perp$ is also a maximal chain of closed subspaces of  $\mathcal H$. This proves that $A$ and $B$ are simultaneously triangularizable.

If $A$ and $B$ are normal, then $A^*$ and $B^*$ are simultaneously diagonalizable, by Theorem \ref{normalnost}. The conclusion of the theorem now immediately follows.
\end {proof}

{\it Acknowledgments.} This work was supported in part by grant P1-0222 of Slovenian Research Agency.

\bigskip
		
\noindent
     Roman Drnov\v{s}ek, Marko Kandi\'{c} : \\
     Faculty of Mathematics and Physics \\
     University of Ljubljana \\
     Jadranska 19 \\
     1000 Ljubljana \\
     Slovenia \\[1mm]
     E-mails : roman.drnovsek@fmf.uni-lj.si, marko.kandic@fmf.uni-lj.si

\end{document}